\documentclass[reqno,10pt]{amsart}
\usepackage{xspace}
\usepackage[bookmarksnumbered,colorlinks]{hyperref}
\usepackage{graphics}

\newcommand{\bt}{\begin{theorem}}                               
\newcommand{\et}{\end{theorem}}                                 
\newcommand{\bd}{\begin{definition}}                            
\newcommand{\ed}{\end{definition}}                              
\newcommand{\bl}{\begin{lemma}}                                 
\newcommand{\el}{\end{lemma}}                                   
\newcommand{\bpr}{\begin{proposition}}                  
\newcommand{\epr}{\end{proposition}}                    
\newcommand{\bere}{\begin{remark}}                      
\newcommand{\ere}{\end{remark}}                                 
\newcommand{\beq}{\begin{equation}}
\newcommand{\eeq}{\end{equation}}
\def\bal#1\eal{\begin{align}#1\end{align}}                      
\def\baln#1\ealn{\begin{align*}#1\end{align*}}          
\def\bml#1\eml{\begin{multline}#1\end{multline}}        
\def\bmln#1\emln{\begin{multline*}#1\end{multline*}}  
\def\bga#1\ega{\begin{gather}#1\end{gather}}
\def\bgan#1\egan{\begin{gather*}#1\end{gather*}}
\newcommand{\de}{\mathrm{d}}

\newcommand{\R}{\ensuremath{\mathbb{R}}\xspace}     
\newcommand{\eps}{\varepsilon}                      
\newcommand{\To}{\longrightarrow}                   
\newcommand{\inte}{\int_0^1\!\!}
\newcommand{\Om}{\ensuremath{\Omega_{p,q}(M)}\xspace}
\newtheorem{theorem}{Theorem}[section]

\newtheorem{lemma}[theorem]{Lemma}
\newtheorem{proposition}[theorem]{Proposition}
\theoremstyle{definition}
\newtheorem{definition}[theorem]{Definition}
\theoremstyle{remark}
\newtheorem{remark}[theorem]{Remark}
\title[The index of a geodesic in a Randers space]{The index of a geodesic in a Randers space and some remarks about the lack of regularity of the energy functional of a Finsler metric}

\author[E. Caponio]{Erasmo Caponio}
\address{Dipartimento di Matematica, Politecnico di Bari, Via Orabona 4, 
70125, Bari, Italy}
\email{caponio@poliba.it}
\thanks{Supported by M.I.U.R. Research project PRIN07 "Metodi Variazionali e  
Topologici nello Studio di Fenomeni Nonlineari"
}



\date{July 31, 2009}

\begin{document}

\begin{abstract}
In a series of papers (\cite{CaJaMa08,  CaJaMa09, CaJaS09})  the relations existing between the metric properties of  Randers spaces and  the conformal geometry of  stationary Lorentzian manifolds  were discovered and investigated. These relations were called in \cite{CaJaS09} {\em Stationary-to-Randers Correspondence (SRC)}.
In this paper we focus on one aspect of SRC,  the equality  between the index of a geodesic in a Randers space and that of its lightlike lift in the associated conformal stationary spacetime.
Moreover we make some remarks about regularity of the energy functional of a Finsler metric on the infinite dimensional manifold of $H^1$ curves connecting two points, in connection with  
infinite dimensional techniques in Morse Theory.
\end{abstract}
\maketitle
\begin{section}{Introduction}
Let $S$ be a manifold of dimension $n$ and $R=\sqrt{h}+\omega$  be a Randers metric on $S$. To $(S,R)$ we associate a one-dimensional higher manifold $M=S\times\R$ endowed with the bilinear symmetric tensor 
\[
g=h-(\omega-\de t)^2.\]
The condition on the norm of $\omega$ ensuring that $R$ is a positive definite function on $TS$, i.e  $(\omega_p(v))^2< h_p(v,v)$ for all $v\in T_pS$ and for all $p\in S$, makes $g$ a non-degenerate symmetric bilinear form of index $1$, that is a Lorentzian metric on $S\times\R$. 

Let $t$ be the natural coordinate on $\R$. The vector field $\partial_t=\frac{\partial}{\partial t}$ on $S\times\R$ is timelike at any point (i.e. 
$g_p(\partial_t,\partial_t)<0$, for all $p\in M$) and  it is a Killing vector field for $g$. A Lorentzian manifold admitting a timelike Killing vector field is called {\em stationary} (see for instance \cite[p. 119]{Wald84}) and whenever the timelike Killing vector field is irrotational is said {\em static}.

For any fixed $p\in S$, the function $R(p,\cdot)\colon T_pS\to [0,+\infty)$ arises  as the non-negative solution of the equation in the variable $\tau$
\beq\label{null}
h_p(v,v)-(\omega_p(v)-\tau)^2=0.\eeq
Eq. \eqref{null} and $\tau> 0$ are the conditions that a future pointing lightlike vector $(v,\tau)\in T_pS\times\R$ has to satisfy by definition. 

We recall that a Lorentzian manifold $(M,g)$  is said {\em time-oriented} if it admits a smooth timelike vector field $Y$. In particular a stationary Lorentzian manifold is time-oriented by one of its timelike Killing vector field. A vector $v\in T_p M$ is said {\em future pointing} (resp. {\em past pointing}) if $g_p(v,Y)<0$ (resp. $g_p(v,Y)>0$) and {\em lightlike} if $g_p(v,v)=0$. Analogously, a smooth curve $\gamma\colon [a,b]\to M$ is said future pointing, past pointing,   lightlike iff its velocity vector field is future pointing, past pointing, lightlike.
Observe that if $(v,\tau)$ is past pointing and lightlike then $\tau$ is equal to the non-positive solution of \eqref{null} and $-\tau$ is equal to the Randers metric obtained reversing $R$, that is $-\tau=R(p,-v)$.

In analogy with a terminolgy used for static spacetimes (cf. \cite[p. 360]{One83}), a stationary Lorentzian manifold $(M,g)$ is said {\em standard} if it is isometric to a product manifold $S\times\R$ endowed with the metric 
\[g_0+w\otimes\de t+\de t\otimes w-\beta\de t^2,\]
where $g_0$, $w$ and $\beta$ are respectively a Riemannian metric, a one-form and a positive function on $S$.  The conditions defining future pointing lightlike vectors on $(M,g)$ define now the non-negative function on $TS$
\[
R=\sqrt{g_0/\beta+(w/\beta)^2}+w/\beta.
\]
Whatever the one-form $w$ is, the norm of $w/\beta$ with respect to the Riemannian metric 
\beq\label{h}
h=g_0/\beta+(w/\beta)^2\eeq 
is less than $1$ and thus $R$ is a Randers metric. 

Since Eq. \eqref{null}  is invariant under conformal transformations of the metric $g$, the same Randers metric $R$ is associated to the conformal class of $g$. Conversely, a Randers space $(S,R)$ individuates the  conformal standard stationary Lorentzian manifold   $(S\times\R,h-(\omega-\de t)^2)$. 

The bijection between Randers spaces and conformal standard stationary Lorentzian manifolds has been called in \cite{CaJaS09} {\em Stationary-to-Randers correspondence (SRC)} and it has been used in \cite{CaJaMa08} and in \cite{CaJaS09} to study the causal structure of a conformal standard stationary Lorentzian manifold.

One of the basic observation about SRC is that there is a one-to-one correspondence between lightlike geodesics of the conformal standard stationary Lorentzian manifold and the geodesics of the associated Randers space. Going into more details, we mention that lightlike geodesics on a Lorentzian manifold are invariant under conformal changes of the metric in the sense that if $\gamma\colon [0,1]\to M$ is a lightlike geodesic of $(M,g)$ then $\gamma$ is a pregeodesic of $\lambda g$ for any positive function $\lambda$, i. e. there exists a reparametrization $\sigma\colon [0,1]\to [0,1]$ such that $\gamma\circ\sigma$ is a lightlike geodesic of $(M,\lambda g)$ (see for example \cite[p.14]{MinS08}). We consider now a conformal  standard stationary Lorentzian manifold $(S\times\R, g)$ and we take, as representative  of the class the metric,   $h-(\omega-\de t)^2$, where $h$ is equal to \eqref{h} and $\omega=w/\beta$.
If $z(s)=(x(s),t(s))$ is a future pointing lightlike geodesic  of $(S\times\R,h-(\omega-\de t)^2)$ then (see \cite[Theorem 4.5]{CaJaMa08} $x$ is a geodesic of the Randers space $(S, R)$, $R=\sqrt{h}+\omega$, parametrized with $h(\dot x,\dot x)=\mathrm{const.}$ The fact that $x$ has to be parametrized with constant Riemannian speed  can be  seen   recalling that $g(\dot z,\dot z)=0$
 and, since  $\partial_t$ is a Killing vector field, $g(\dot z,\partial_t)=\omega(\dot x)-\dot t=\mathrm{const.}$ thus also $h(\dot x,\dot x)$ has to be constant. 
 
The other way round, a geodesic $x=x(s)$ in $(S,R)$ can be lifted to a future pointing lightlike curve on $S\times\R$ by taking 
\beq\label{t}
t=t(s)=t_0+\int_{s_0}^s R(x,\dot x).\eeq
If $x$ is parametrized with constant Riemannian speed, its  future pointing  lightlike lift is a lightlike geodesic of $(S\times\R, 
h-(\omega-\de t)^2)$. 

The same relation holds between geodesics of the reversed metric $\tilde R(x,v)=R(x,-v)$ and past pointing lightlike geodesic 
of $(S\times\R, 
h-(\omega-\de t)^2)$.

In Section 2 of this note, we  focus on one aspect of SRC that is the equality between the  index of a geodesic in the Randers space $(S, R)$ and the index of its future pointing lightlike lift in $(S\times\R,   h-(\omega-\de t)^2)$.

An immediate consequence of this equality (which holds also for a geodesic of the reversed Randers metric $\tilde R$ and the corresponding past pointing lightlike geodesic of $(S\times\R, 
h-(\omega-\de t)^2)$),
 is that the index of a lightlike geodesic is a conformal invariant for standard stationary Lorentzian manifolds.    
This gives an alternative proof to a well known fact which holds for any conformal Lorentzian manifold (see for example \cite[Theorem 2.36]{MinS08}).

Another consequence of this equality is that the Morse theory for future pointing lightlike geodesic connecting a point $\tilde p=(p,t_0)$ to an integral line  of the timelike Killing vector field $\partial_t$ passing through the point $\tilde q=(q,t_0)$,
can be reduced to the Morse theory for geodesics  connecting the points $p$ and $q$ in the associated Randers space.

Altough Morse theory for geodesics connecting two points on a Finsler manifold $(M,F)$ can be developed by using finite dimensional approximations of the path space  by broken geodesics (see \cite{Shen01}), infinite dimensional techniques in Morse theory  can be adapted to work in the Sobolev manifold $\Omega_{p,q}(M)$ of the $H^1$ curves connecting the points $p$ and $q$. The main problem in regard to this approach is the lack of twice Frechet differentiability of the energy functional $E$ of a Finsler metric at any critical point with respect to the $H^1$--topology. Anyway $E$ has enough regularity to get a version of the generalized Morse Lemma which allows us to compute the critical groups and to obtain the Morse relations (see \cite{CaJaMa09}).   In Section 3 we illustrate what is the problem in trying to prove that $E$ is twice Frechet differentiable with respect to the $H^1$--topology and we will extend to the Finsler case a recent argument by A. Abbandondandolo and M. Schwartz \cite{AbbSch08}. In fact, in \cite{AbbSch08} the authors prove that a smooth time dependent Lagrangian $L\colon [0,1]\times TM\to \R$, which is subquadratic in the velocities and whose action  functional is twice Frechet differentiable at a regular curve  on the Sobolev  manifold $\Omega(M)$ of all the $H^1$ curves on $M$, must be a polynomial of degree at most two in the velocity variables along the curve.
This  fact can be seen as an infinite dimensional version of the well known property that if the square of a Finsler metric is $C^2$ on the whole $TM$ then actually it  is the square of the norm of a Riemannian metric. 
\end{section}
\begin{section}{The equality between the indexes}
Let $M$ be a Lorentzian or a Finsler manifold and let $\gamma$ be a geodesic on $M$.
By $\mu(\gamma)$ we denote the {\em index} of $\gamma$, that is the number of conjugate points along $\gamma$ counted with their multiplicity.  The equality between $\mu(x)$, where $x$ is a geodesic 
of the Randers space $(S,R)$, and $\mu(z)$, where $z$ is the future pointing lightlike lift of $x$ in $(S\times\R,g=h-(\omega-\de t)^2)$, can be carried out by comparing the Jacobi equation of $x$ in $(S,R)$ with the Jacobi equation of $z$ in $(S\times\R,h-(\omega-\de t)^2)$, as done in \cite[Theorem 13 ]{CaJaMa09}.

Here we give a different proof based on a  comparison of the Morse index  of the energy functional of the Randers metric at $x$ 
and the Morse index at $z$ 
of the  functional introduced by Uhlenbeck in \cite{Uhlenb75}:
\[
J(\sigma)=\inte\big(g(\dot \sigma,\dot \sigma)+\big(\tfrac{\de P(\sigma)}{\de s}\big)^2\big)\de s.\]
Here $\sigma$ belongs to the set of piecewise differentiable  curves on $S\times\R$,
satisfying the constraint $g(\dot \sigma,\dot \sigma)=0$  and the boundary conditions $\sigma(0)=\tilde p\in S\times\R$, $\sigma(1)\in l(\R)$, where $l=l(s)$ is an integral line of the Killing vector field $\partial_t$ ($\tilde p\not\in l(\R)$) and $P\colon     S\times\R\to \R$ is the natural projection on $\R$.

The critical point of $J$ are the lightlike geodesics connecting $\tilde p$ to $l(\R)$.
Moreover $J$ admits second variation at any critical point.
A critical point is non degenerate if and only if its endpoints are non-conjugate. The Morse index of a critical point is finite and it is equal to $\mu(z)$ (see \cite[Lemma 4.2]{Uhlenb75}).
Using these properties of $J$ we can prove the following 
\begin{theorem}
Let $(S\times\R,h-(\omega-\de t)^2)$ be the conformal standard stationary spacetime associated by SRC to $(S,R)$ and   $z(s)=(x(s),t(s))\colon[0,1]\to S\times\R$ be the future pointing lightlike geodesic associated to the geodesic $x(s)$ in $(S,R)$. 
Then 
the points $x(0)$ and $x(1)$ are non-conjugate along $x$ in $(S,R)$ if and only if the points $z(0)$ and $z(1)$ are non-conjugate along 
$z$ in $(S\times\R,h-(\omega-\de t)^2)$.  Moreover
\[\mu(z)=\mu(x).\]
\end{theorem}
\begin{proof}
Consider the energy functional of the Randers metric $R$ \[E(\gamma)=\tfrac{1}{2}\int_0^1\!\!R^2(\gamma,\dot \gamma)\de s.\]
Since the Morse index of $E$ at the geodesic $x$ is equal to $\mu(x)$
(see \cite{Matsum86}) and the Morse index of $J$ at $z$ is equal to $\mu(z)$, it is enough to prove the equality for the Morse indexes.
To this end, we will show that the set $\mathcal W_x$ of continuous piecewise smooth vector field along $x$ vanishing at $x(0)$ and $x(1)$  is isomorphic 
to the set of admissible variations $\mathcal U_z$ for $J$ which is given by the continuous piecewice smooth vector fields $U$ along $z$, vanishing at $z(0)$ and $z(1)$ and such that $g(\dot z,U)=0$ (see \cite{Uhlenb75}).
Let us denote by $\mathcal P_{x(0),x(1)}(S)$ and $\mathcal{L}_{z(0),l}(S\times\R)$ respectively the set of the continuous, piecewise smooth curves on $S$, parametrized on the interval $[0,1]$ and connecting $x(0)$ to $x(1)$ and the set of the continuous, piecewise smooth, future pointing, lightlike curves on $S\times\R$, parametrized on $[0,1]$ and connecting $z(0)$ to $l(\R)$.
Consider the map
\[
\Psi(\gamma)(s)=\left(\gamma(s),t_0+\int_0^s R(\gamma,\dot \gamma)\de \nu\right).\]
Recalling that the future pointing lightlike lift of a curve $\gamma$ in $S$ has $t$ component in $S\times\R$ given by \eqref{t}, we immediately see that 
$\Psi$ maps $\mathcal P_{x(0),x(1)}(S)$ to $\mathcal{L}_{z(0),l}(S\times\R)$.

We are going to show that the isomorphism between $\mathcal W_x$ and $\mathcal U_z$ is given by $\Psi'(x)$ where, for each $W\in \mathcal W_x$, $\Psi'(x)[W]$ is the vector field along $z$ belonging to $\mathcal U_z$ defined as
$\frac{\partial}{\partial r}(\Psi\circ\varphi_0)(r,s)|_{r=0}$ where $\varphi_0=\varphi_0(r,s)\colon (-\eps,\eps)\times[0,1]\to S$ is the variation of the geodesic $x$ defined by
$W$.
Observe that, since $x$ is a critical point of the length functional $x\mapsto\int_0^1R(x,\dot x)\de s$, for any $W\in \mathcal W_x$ there holds
\[\left(\Psi'(x)[W]\right)(0)=\left(\Psi'(x)[W]\right)(1)=0.\]

Let $I$ be the functional defined in the same way as $J$  
\[
I(\sigma)=\inte\big(g(\dot \sigma,\dot \sigma)+\big(\tfrac{\de P(\sigma)}{\de s}\big)^2\big)\de s\]
but now $\sigma$ varies on the set of the continuous, piecewise smooth, future pointing curves, non necessarily lightlike,  connecting $z(0)$ to $l(\R)$.
For any future pointing lightlike curve $\sigma(s)=(\gamma(s),\tau(s))$ we have
\[J(\sigma)=I(\Psi(\gamma))=2E(\gamma).\]
Moreover, for any geodesic $x$ of $(S,R)$ and for any $W\in \mathcal W_x$, we have
\beq\label{depsi}
\left(\Psi'(x)[W]\right)(s)=\left(W(s),\int_0^s(R_x(x,\dot x)[W]+R_v(x,\dot x)[\dot W])\de s\right);\eeq
hence $\Psi'(x)$ is an injective map (notice that in the right-hand side of \eqref{depsi}, with an abuse of notation, we have used the expression $R_x(x,\dot x)[W]+R_v(x,\dot x)$ which is meaningful only in local coordinates).

Let $U(s)=(W(s),u(s))\in \mathcal U_z$. We are going to show that $\Psi'(x)[W]=U$ and hence $\Psi'(x)$ is also surjective.

As $U\in \mathcal U_z$, we have  
\[g(U,\dot z)=0\quad\Leftrightarrow\quad h(W,\dot x)-(\omega(W)-u)(\omega(\dot x)-\dot t) =0.\]
Since  $z$ is lightlike and future pointing $\omega(\dot x)-\dot t=-\sqrt{h(\dot x,\dot x)}$ and thus
\[u =\frac{h(W,\dot x)}{\sqrt{h(\dot x,\dot x)}}+\omega(W)\]
Since $x$ is a critical point of $E$ and $W(0)=0$, integrating by part the $t$ component of the vector field $\Psi'(x)[W]$ in \eqref{depsi} and using the Euler-Lagrange equation satisfied by $x$, we deduce that such a component is equal to 
\[R_v(x,\dot x)[W]=\frac{h(W,\dot x)}{\sqrt{h(\dot x,\dot x)}}+\omega(W)=u.\]

Now let  $\varphi=\varphi(r,s)\colon(-\eps,\eps)\times[0,1]\to S\times\R$ be a variation defined by the admissible variational vector field $U=(W,u)$, and $\varphi_0=\varphi_0(r,s)\colon(-\eps,\eps)\times[0,1]\to S$ be the one defined by $W$, we have that 
\bmln
J''(z)(U,U)=\frac{\de^2}{\de r^2}J(\varphi(r,\cdot))_{\big| r=0}\\
=\frac{\de^2}{\de r^2}I(\Psi(\varphi_0(r,\cdot)))_{\big| r=0}=
2\frac{\de^2}{\de r^2}E(\varphi_0(r,\cdot))_{\big| r=0}=2 E''(x)(W,W).\emln	
By polarization, the above equality gives the thesis.
\end{proof}
\end{section}
\begin{section}{The lack of twice differentiability of the energy functional with respect to  the $H^1$--topology}
Let $(M,F)$ be a Finsler manifold and $p,\ q\in M$.
Let $\Omega(M)$ be the Sobolev manifold of the absolutely continuous curves $\gamma\colon[0,1]\to M$, whose square of the norm of the velocity vector field is integrable with respect to a fixed (and then to any) auxiliary Riemannian metric $\alpha$ on $M$. Let us denote by \Om the submanifold  of the curves in $\Omega(M)$, such that $\gamma(0)=p$, $\gamma(1)=q$ (see \cite{Klinge82}). Let us consider the energy functional of $F$ on \Om: 
\[E\colon \Om\to\R,\quad E(\gamma)=\tfrac{1}{2}\inte F^2(\gamma,\dot \gamma)\de s\]
It is well known that $E$ is  $C^{1,1}$ on $\Om$,  \cite{Mercur77}.

We are going to  show that if
$E$ is twice differentiable on \Om at a regular curve $\gamma$ then  $F^2$ is the square of the norm of a Riemannian metric along the curve. 

By {\em regular curve}  we mean a curve $\gamma\in \Om$ such that $\dot\gamma\neq 0$ $\mathrm{a.\ e.}$ in $[0,1]$.

\bere
We point out that in
\cite{AbbSch08} the authors consider a time-dependent Lagrangian  $L\colon [0,1]\times TM\to \R$, $L=L(t,q,v)$,  which is $C^2$ on $TM$ and which satisfies  the following conditions: there exists a continuous positive function $C=C(q)$ such that for any $(t,q,v)\in [0,1]\times TM$:
\baln &\|\partial_{vv}L(t,q,v)\|\leq C(q),\\
&\|\partial_{vq}L(t,q,v)\|\leq C(q)(1+\sqrt{\alpha(v,v)}),\\
&\|\partial_{qq}L(t,q,v)\|\leq C(q)(1+\alpha(v,v)).\ealn
They prove that if the action functional of $L$
\[\gamma\colon \Omega(M)\to\inte L(t,\gamma(t),\dot\gamma(t))\de t,\]
is twice differetiable in $\Omega(M)$ at a curve $\gamma$, then the map
\[v\in T_{\gamma(t)}M\mapsto L(t,\gamma(t),v)\] 
is a polynomial of degree at most two.
Thus, in particular, the subquadratic and strongly convex in the velocities, time-independent, $C^2$ Lagrangians whose action functional is twice differentiable at any curve in $\Omega(M)$ are all and only of the type
\[L(q,v)=h_q(v,v)+\omega_q(v)+V(q),\]
where $h$, $\omega$ and $V$ are respectively a Riemannian metric, a one-form and a function on $M$.
Clearly, the square of a Finsler metric satisfies the growth conditions above
but it is only a $C^{1,1}$ function on $TM$ (it is $C^2$ on $TM\setminus 0$). Anyway, as we show below, the proof in \cite{AbbSch08} does not involve existence and continuity of the derivatives $\partial_{vv}L(t,q,v)$ for $v=0$ and then it extends also to the Finsler case.
Another difference from \cite{AbbSch08} is that we consider the manifold \Om and not $\Omega(M)$.
\ere
Before going into the details of the proof, we would like to point out  
what is the problem in trying to prove that $E$ is twice differentiable in \Om at a regular curve.
To fix ideas, we assume that $F$ is defined on an open subset $U$ of $\R^n$,  $F\colon TU\to\R$, $U\subset\R^n$.
Arguing as in \cite[Proposition 3.1]{AbbSch08} gives that $E$ is twice Gateaux differentiable in $\Omega_{p,q}(U)$  at any regular curve $x$ and its second Gateaux differential is equal to
\baln
\lefteqn{D^2\tilde E(x)[\xi,\eta]=}&\nonumber\\
&=\frac12\inte\big(\partial_{qq}F^2(x,\dot x)[\xi,\eta]+\partial_{vq}F^2(x,\dot x)[\dot \xi, \eta]\big)\de s\nonumber\\
&\quad+\frac 12\inte\big(\partial_{qv}F^2(x,\dot x)[\xi,\dot\eta]\big)\de s+\partial_{vv}F^2(x,\dot x)[\dot \xi,\dot \eta]\big)\de s.\ealn
The problem is the continuity  of the map 
\[x\in \Omega_{p,q}(U)\mapsto \inte\partial_{vv}F^2(x,\dot x)[\cdot,\cdot]\de s,\]
where the target space is the space of bounded bilinear operators on $H^1_0([0,1],U)$.
Namely, we can prove that  if 
$x_n\to x$ in $\Omega_{p,q}(U)$  then 
\[\int_0^1  \partial_{vv} F^2(x_n, \dot x_n)[\dot \xi, \dot\eta]\de s \To \int_0^1  \partial_{vv}F^2(x, \dot x)[\dot \xi, \dot\eta]\de s,\quad \text{as $n\to +\infty$,}\] 
but we cannot prove that the convergence is  uniform with respect to $\xi$ and $\eta$ in the unit ball of $H^1_0([0,1],U)$,
unless $\partial_{vv}F^2$ is independent from $v$ (and then $F^2$ is the square of the norm of a  Riemannian metric).
In fact, we have the following
\bpr
If the energy functional of a Finsler metric $F$ 
is twice differentiable at a regular
curve $\gamma\in \Om$ then, for $\mathrm{a.\ e.}$ $s\in [0,1]$, the function
\[
v\in T_{\gamma(s)} M\mapsto F^2(\gamma(s),v)
\]
is a quadratic positive definite form.
\epr
\begin{proof}
For simplicity and without loss of generality, we  prove the statement in the case where $M$ is an open subset of $\R^n$.
Since $\dot\gamma\neq 0$ $\mathrm{a.\ e.}$ on $[0,1]$, the thesis is equivalent to the fact that, for almost every $s\in [0,1]$, there holds
\[
\partial_v F^2(\gamma(s),\dot \gamma(s)+v) - \partial_v F^2 (\gamma(s),\dot\gamma(s)) - \partial_{vv}F^2(\gamma(s),\dot\gamma(s)) [v] = 0, \]
for all $v\in \R^n$.
By contradiction, we assume that there is a set of positive measure $J\subset [0,1]$ and  two non-zero vectors $v,w\in \R^n$, and a positive number $c$ such that    
\beq\label{c}
\Big(\partial_v F^2(\gamma(s),\dot \gamma(s)+v) - \partial_v F^2 (\gamma(s),\dot\gamma(s)) - \partial_{vv}F^2(\gamma(s),\dot\gamma(s)) [v]\Big)\cdot w>c, \eeq
For every $\epsilon>0$ smaller than the measure of $J$, choose a subset $J_{\epsilon}\subset J$ of measure $\epsilon$, in such a way that $J_{\epsilon} \subset J_{\epsilon'}$ if $\epsilon<\epsilon'$. Define the following  functions 
\[
\eta_{\epsilon} (s) = v \int_0^s (\chi_{\epsilon}(t)- \epsilon)\, dt, \quad 
\xi_{\epsilon} (s) = w \int_0^s (\chi_{\epsilon}(t)- \epsilon)\, dt,
\]
where $\chi_{\epsilon}$ is the characteristic function of $J_{\epsilon}$.
Observe that, for any $\eps$, the functions $\eta_{\epsilon}, \xi_{\epsilon}$ belong to $T_{\gamma}\Om=H^1_0([0,1],\R^n)$
and 
\[\|\eta_{\epsilon}\|_{H^1_0}=|v|(\epsilon-\epsilon^2)^{1/2} \quad \|\xi_{\epsilon}\|_{H^1_0}=|w|(\epsilon-\epsilon^2)^{1/2}.\]
We can repeat  the proof of Proposition~3.2 in \cite{AbbSch08} taking care only that  the derivatives of $\eta$ and $\xi$ here are given by  $v(\chi_{\epsilon}-\epsilon)$ and $w(\chi_{\epsilon}-\epsilon)$ and  the terms  involving integrals of the type
\[\inte \big(\partial_vF^2(\gamma+\eta_{\epsilon},\dot\gamma+\dot\eta_{\epsilon})-\partial_vF^2(\gamma,\dot\gamma+\dot\eta_{\epsilon})-\partial_{qv}F^2(\gamma,\dot\gamma+\dot\eta_{\epsilon})[\eta_{\epsilon}]\big)\cdot (\epsilon\, w)\,\de s\]
belong to $o(\epsilon)$, as $\epsilon\to 0$ (such terms do not appear in \cite{AbbSch08} because   the functions playing  the role of $\eta$ and $\xi$, not having to  belong to $H^1_0([0,1],\R^n)$, are defined as $\eta(s)= v \int_0^s \chi_{\epsilon}(t)\de t$ and $\xi(s)= w \int_0^s \chi_{\epsilon}(t)\de t$).

We point out that the non existence of the derivatives $\partial_{vv}F^2(q,v)$  for $v=0$ does not affect that part of the proof since  only the smoothness of $\partial_v F^2(q,v)$ with respect to $q$ is used. 

Thus, as in \cite{AbbSch08}, we can deduce
\beq\label{fine}
\int_0^1 \left( \partial_v F^2 (\gamma,\dot\gamma+\dot\eta_{\epsilon}) - \partial_v F^2(\gamma,\dot\gamma) - \partial_{vv}F^2(\gamma,\dot\gamma) [\dot\eta_{\epsilon}] \right)\cdot \dot\xi_{\epsilon} \de s = o(\epsilon), \quad \text{as $\epsilon \to 0$}.
\eeq
The left-hand side of \eqref{fine} is equal to
\bal
\lefteqn{\int_0^1 \big(\partial_v F^2 (\gamma,\dot\gamma+\dot\eta_{\epsilon}) - \partial_v F^2(\gamma,\dot\gamma) - \partial_{vv}F^2(\gamma,\dot\gamma) [\dot\eta_{\epsilon}] \big)\cdot (\chi_{\epsilon}w-\epsilon w)\de s=}&\nonumber\\
&=\int_{J_{\epsilon}} \big(\partial_vF^2(\gamma,\dot\gamma+(1-\epsilon) v)
-\partial_vF^2(\gamma,\dot\gamma)- \partial_{vv}F^2(\gamma,\dot\gamma)[(1-\epsilon)v]\big)\cdot w\de s\nonumber\\
&\quad+\inte\big(\partial_v F^2 (\gamma,\dot\gamma+\dot\eta_{\epsilon}) - \partial_v F^2(\gamma,\dot\gamma) - \partial_{vv}F^2(\gamma,\dot\gamma) [\dot\eta_{\epsilon}] \big)\cdot (\epsilon w)\de s\label{fine2}
\eal
Since $\dot\eta_{\epsilon}\to 0$ $\mathrm{a.\ e.}$ as $\epsilon\to 0$, by the Lebesgue's dominated convergence theorem, the absolute value of the second integral in the right-hand side of \eqref{fine2} is less than
$\epsilon|w|o(1)$.

Therefore, putting together \eqref{fine} and \eqref{fine2}, we have
\beq\label{fine3}
\int_{J_{\epsilon}} \!\!\big(\partial_vF^2(\gamma,\dot\gamma+(1-\epsilon) v)
-\partial_vF^2(\gamma,\dot\gamma)- \partial_{vv}F^2(\gamma,\dot\gamma)[(1-\epsilon)v]\big)\cdot w\de s
= o(\epsilon),\quad\text{as $\epsilon \to 0$}.
\eeq
By \eqref{c} and the continuity of the map 
\[v\in\R^n\mapsto \Big(\partial_v F^2(\gamma(s),\dot \gamma(s)+v) - \partial_v F^2 (\gamma(s),\dot\gamma(s)) - \partial_{vv}F^2(\gamma(s),\dot\gamma(s)) [v]\Big)\cdot w\]
the integral in \eqref{fine3} is larger than $c \epsilon$, for $\epsilon$ small enough,  giving a contradiction.
\end{proof}

\end{section}

\end{document}